\documentclass[a4paper,11pt]{amsart}

\usepackage[top=30truemm,bottom=25truemm,left=35truemm,right=35truemm]{geometry}

\usepackage{amssymb}
\usepackage{url} 

\usepackage{amscd} 
\usepackage[all]{xy} 
\usepackage{graphics}
\usepackage{multienum}

\newtheorem{thm}{Theorem}[section]
\newtheorem{lem}[thm]{Lemma}
\newtheorem{lem-dfn}[thm]{Lemma-Definition}
\newtheorem{prop}[thm]{Proposition}

\theoremstyle{definition}
\newtheorem{defn}[thm]{Definition}

\newtheorem{prob}[thm]{Problem}

\newtheorem*{acknowledgement}{Acknowledgement}

\theoremstyle{remark}
\newtheorem{claim}[thm]{Claim}
\newtheorem{rem}[thm]{Remark}

\begin{document}



\title[A characterization of binomial Macaulay dual generators for c.i]{A characterization of binomial Macaulay dual generators for  complete intersections}

\author{Kohsuke Shibata}
\address{Department of Mathematics, School of Engineering, 
Tokyo Denki University, Adachi-ku, Tokyo 120-8551, Japan.}
\email{shibata.kohsuke@mail.dendai.ac.jp}

\begin{abstract}
We characterize a binomial such that the Artinian algebra whose Macaulay dual generator is the binomial is a complete intersection.
As an application, we prove that the Artinian algebra with a binomial Macaulay dual generator has the strong Lefschetz property in characteristic 0 if the Artinian algebra is a complete intersection.
\end{abstract}

\maketitle

\section{Introduction}
Throughout this paper, $k$ is a field.
Let $R=k[x_1,\dots,x_N]$ and $S=k[X_1,\dots,X_N]$ be polynomial rings.
Let us consider the action $\circ$ of $R$ on $S$ given by
\[
x_1^{a_1}\cdots x_N^{a_N}\circ X_1^{b_1}\cdots X_N^{b_N}=
  \begin{cases}
    X_1^{b_1-a_1}\cdots X_N^{b_N-a_N} & \text{if $b_i\ge a_i$ for any $i$,} \\
    0                 & \text{othewise}
  \end{cases},
\]
for any $a_1,\dots,a_N,b_1,\dots, b_N\in \mathbb Z_{\ge 0}$.
Then $S$ is an $R$-module via the action $\circ$.
For $F\in S\setminus\{0\}$, we define the $(x_1,\dots, x_N)$-primary ideal $\operatorname{Ann}_R(F)$ by
\[
\operatorname{Ann}_R(F)=\{f\in R \mid f\circ F=0\}.
\]
Then $R/\operatorname{Ann}_R(F)$ is an Artinian Gorenstein ring for any $F\in S\setminus\{0\}$.
Moreover, if $R/I$ is an Artinian Gorenstein ring with $\sqrt{I}=(x_1,\dots, x_N)$, then there exists $F\in S$ such that $I=\operatorname{Ann}_R(F)$.
Since any complete intersection ring is Gorenstein, it is natural to consider the following problem.

\begin{prob}[{\cite[Chapter 9.L]{IK}}]
Characterize  $F\in S$ such that $R/\operatorname{Ann}_R(F)$ is a complete intersection.
\end{prob}

Any Gorenstein local ring of embedding dimension at most two is a complete intersection.
Hence if $N=2$, $R/\operatorname{Ann}_R(F)$ is a complete intersection for any $F\in S\setminus\{0\}$.
In \cite{HWW}, Harima, Wachi and Watanabe give  necessary and sufficient conditions for $R/\operatorname{Ann}_R(F)$ to be a complete intersection when $F$ is a homogeneous of degree $N$.
In \cite{E}, Elias characterize that $R/\operatorname{Ann}_R(F)$ is a complete intersection for a homogeneous $F$ using a regular sequence on $R$.
If $F$ is a monomial, then $R/\operatorname{Ann}_R(F)$ is a complete intersection.
In this paper, we consider this problem for binomials, which is the next simplest case after monomials.

\begin{prob}[{\cite[Section 4]{ADMMSV}, \cite[Problem 3.6]{HWW}}]\label{problem binomial}
Characterize a binomial $F\in S$ such that $R/\operatorname{Ann}_R(F)$ is a complete intersection.
\end{prob}

In \cite{ADFMMSV}, Altafi, Dinu, Faridi, Masuti, Mir\'{o}-Roig,  Seceleanu, and Villamizar solve this problem when $N=3$,  $F$ is homogeneous and $k$ is an algebraically closed field.
In this paper, we solve Problem \ref{problem binomial}.
Moreover, we determine $\operatorname{Ann}_R(F)$ for a binomial $F$ if $R/\operatorname{Ann}_R(F)$ is a complete intersection.

\begin{thm}\label{thm in intro}
Let $m,n\in \mathbb N$.
Let $R=k[x_1,\dots,x_{m+n}]$ and $S=k[X_1,\dots,X_{m+n}]$ be  polynomial rings.
Let \[F=X_1^{a_1}\cdots X_{m+n}^{a_{m+n}}(c_1X_1^{b_1}\cdots X_m^{b_m}-c_2X_{m+1}^{b_{m+1}}\cdots X_{m+n}^{b_{m+n}} )\in S\] 
be a binomial, where $a_1,\dots,a_{m+n},b_1\dots,b_{m+n}\in \mathbb Z_{\ge 0}$ and $c_1,c_2\in k\setminus \{0\}$.
Let 
\[d_1=\#\{i \mid b_i\neq 0,\ i=1,\dots, m\},\quad d_2=\#\{i \mid b_i\neq 0,\ i=m+1,\dots, m+n\}.\] 
Suppose that $d_1\ge d_2\ge 1$, \[b_{d_1+1}=b_{d_1+2}=\cdots=b_{m}=0,\quad b_{m+d_2+1}=b_{m+d_2+2}=\cdots=b_{m+n}=0.\]
Let \[v=\min\{i\in \mathbb N \mid (x_1^{b_1}\cdots x_m^{b_m})^i\circ (X_1^{a_1}\cdots X_m^{a_m})=0\}.\]
Then 
\begin{enumerate}
\item $R/\operatorname{Ann}_R(F)$ is a complete intersection if and only if one of the following conditions holds:

\begin{enumerate}
\item $d_1=d_2=1$.
\item $d_1\ge 2$, $d_2=1$ and $a_{m+1}+1\ge vb_{m+1}$.
\end{enumerate}

\item Suppose that $d_1=d_2=1$. Let $w=\min\{i\in \mathbb N \mid a_{m+1}+1\le ib_{m+1}\}$ and $I=(x_2^{a_2+1},\cdots, x_m^{a_m+1},x_{m+2}^{a_{m+2}+1},\cdots, x_{m+n}^{a_{m+n}+1})$.
\begin{enumerate}
\item If $v<w$, then 
\[\operatorname{Ann}_R(F)=(x_1^{a_1+b_1+1},\sum_{i=0}^{v}c_1^{v-i}c_2^ix_1^{ib_1}x_{m+1}^{a_{m+1}+1-ib_{m+1}})+I.\]

\item If $v>w$, then 
\[\operatorname{Ann}_R(F)=(x_{m+1}^{a_{m+1}+b_{m+1}+1},\sum_{i=0}^{w}c_1^{i}c_2^{w-i}x_1^{a_1+1-ib_1}x_{m+1}^{ib_{m+1}})+I.\]

\item If $v=w$, then $\operatorname{Ann}_R(F)=(p,q)+I,$
where \[p=\sum_{i=0}^{v}c_1^{v-i}c_2^ix_1^{ib_1}x_{m+1}^{(v-i)b_{m+1}},\quad q=\sum_{i=0}^{v-1}c_1^{v-1-i}c_2^ix_1^{a_1+1-(v-1-i)b_1}x_{m+1}^{a_{m+1}+1-ib_{m+1}}.\]
\end{enumerate}

\item Suppose that $d_2=1$ and $a_{m+1}+1\ge vb_{m+1}$. 
Then
\[
\operatorname{Ann}_R(F)=(x_1^{a_1+b_1+1}, \dots, x_m^{a_m+b_m+1}, p,x_{m+2}^{a_{m+2}+1},\cdots, x_{m+n}^{a_{m+n}+1}),
\]
where $p=\sum_{i=0}^{v}c_1^{v-i}c_2^i(x_1^{b_1}\cdots x_m^{b_m})^ix_{m+1}^{a_{m+1}+1-ib_{m+1}}$.
\end{enumerate}

\end{thm}

For a graded Artinian algebra  $A$ over $k$, we say that
$A$ has the strong Lefschetz property if there exists $z\in A_1$ such that  the multiplication map $\times z^d: A_i \to A_{i+d}$ has maximal rank for any $i,d\in \mathbb Z_{\ge 0}$.
In \cite{RRR}, Reid, Roberts and Roitman conjectured that any graded Artinian complete intersection ring has the strong Lefschetz property if  $\operatorname{char}k=0$.
This conjecture holds if $A$ is a monomial complete intersection (\cite{S},\cite{W85}) or $A$ is an Artinian ring of embedding dimension two (\cite{HMNW}).
By the conjecture for monomial complete intersections, it follows that the conjecture holds for $R/\operatorname{Ann}_R(F)$, where $F$ is a monomial.
In this paper, we consider the following problem.

\begin{prob}\label{SLP problem for binomial}
Does $R/\operatorname{Ann}_R(F)$ have the strong Lefschetz property for a homogeneous binomial $F$ if $\operatorname{char}k=0$ and $R/\operatorname{Ann}_R(F)$ is a complete intersection?
\end{prob}

Problem \ref{SLP problem for binomial} is solved  positively when $N=3$ and $k$ is an algebraically  closed field (\cite{ADFMMSV}), or when $F=X_1^{a_1}\cdots X_{N}^{a_{N}}(m_1-m_2)$, where $m_1$ and $m_2$ are monomials in the variables $X_1,X_2,X_3$ (\cite{ADMMSV}).
In this paper, as an application of Theorem \ref{thm in intro}, we prove that Problem \ref{SLP problem for binomial} is affirmative.

This paper is organized as follows.
In Section 2, we characterize a binomial $F\in S$ such that $R/\operatorname{Ann}_R(F)$ is a complete intersection.
In Section 3, we prove that $R/\operatorname{Ann}_R(F)$ has the strong Lefschetz property for a homogeneous binomial $F$ if $\operatorname{char}k=0$ and $R/\operatorname{Ann}_R(F)$ is a complete intersection.

\begin{acknowledgement}
The  author is partially supported by JSPS KAKENHI No. 19K14496 and 23K12958.

\end{acknowledgement}

\section{binomial Macaulay dual generator for c.i}   
In this section, we characterize a binomial such that the Artinian algebra whose Macaulay dual generator is the binomial is a complete intersection.

First, we will define an action on a polynomial ring, and an annihilator of a polynomial using the action. 
\begin{defn}
Let $R=k[x_1,\dots,x_N]$ and $S=k[X_1,\dots,X_N]$ be polynomial rings.
\begin{enumerate}
\item
We define the action $\circ$ of $R$ on $S$ by
\[
x_1^{a_1}\cdots x_N^{a_N}\circ X_1^{b_1}\cdots X_N^{b_N}=
  \begin{cases}
    X_1^{b_1-a_1}\cdots X_N^{b_N-a_N} & \text{if $b_i\ge a_i$ for any $i$,} \\
    0                 & \text{othewise}
  \end{cases},
\]
for any $a_1,\dots,a_N,b_1,\dots, b_N\in \mathbb Z_{\ge 0}$.

\item
For $F\in S\setminus\{0\}$, we define the $(x_1,\dots, x_N)$-primary ideal $\operatorname{Ann}_R(F)$ by
\[
\operatorname{Ann}_R(F)=\{f\in R \mid f\circ F=0\}.
\]

\item
Let $A=R/I$ be an Artinian Gorenstein ring.
We say $F\in S$ is a Macaulay dual generator of $A$ if $I=\operatorname{Ann}_R(F)$.

\end{enumerate}
\end{defn}


\begin{rem}
$R/\operatorname{Ann}_R(F)$ is an Artinian Gorenstein ring for any $F\in S\setminus\{0\}$. Moreover, if $R/I$ is an Artinian Gorenstein ring with $\sqrt{I}=(x_1,\dots, x_N)$, then there exists $F\in S$ such that $I=\operatorname{Ann}_R(F)$ (see \cite[Theorem 21.6]{Eis}, \cite[Theorem 2.1]{HW}).
\end{rem}

For $f=\sum_{(i_1,\dots,i_N)\in \mathbb Z_{\ge 0}^N}c_{i_1,\dots,i_N}x_1^{i_1}\cdots x_N^{i_N}$\ $(c_{i_1,\dots,i_N}\in k)$,
we say $f$ cantains the term $x_1^{i_1}\cdots x_N^{i_N}$ if $c_{i_1,\dots,i_N}\neq 0$.
We write $c_{i_1,\dots,i_N}x_1^{i_1}\cdots x_N^{i_N} \in f$ if $f$ contains the term $x_1^{i_1}\cdots x_N^{i_N}$.

\begin{lem}\label{lem property in Ann}
Let $R=k[x_1,\dots,x_{m+n}]$ and $S=k[X_1,\dots,X_{m+n}]$ be  polynomial rings.
Let \[F=X_1^{a_1}\cdots X_{m+n}^{a_{m+n}}(X_1^{b_1}\cdots X_m^{b_m}-cX_{m+1}^{b_{m+1}}\cdots X_{m+n}^{b_{m+n}})\in S\]
be a binomial, where $a_1,\dots,a_{m+n},b_1\dots,b_{m+n}\in \mathbb Z_{\ge 0}$ and $c\in k\setminus\{0\}$.
Let $F_1$ and $F_2$ be monomials such that $F=F_1-cF_2$, that is 
\[F_1=X_1^{a_1+b_1}\cdots X_m^{a_m+b_m}X_{m+1}^{a_{m+1}}\cdots X_{m+n}^{a_{m+n}}\]
\[ F_2=X_1^{a_1}\cdots X_m^{a_m}X_{m+1}^{a_{m+1}+b_{m+1}}\cdots X_{m+n}^{a_{m+n}+b_{m+n}}.\]
Let $f\in \operatorname{Ann}_R(F)$ and let $g=dx_1^{s_1}\cdots x_{m+n}^{s_{m+n}}$ be a monomial of $R$, where $s_1,\dots,s_{m+n}\in \mathbb Z_{\ge 0}$ and $d\in k\setminus\{0\}$.
Suppose that $g\in f$.
\begin{enumerate}

\item If $g\circ F_1\neq 0$, then $s_{i}\ge b_{i}$ for any $i=1,\dots, m$ and 
\[c^{-1}dx_1^{s_1-b_1}\cdots x_m^{s_m-b_m}x_{m+1}^{s_{m+1}+b_{m+1}} \cdots x_{m+n}^{s_{m+n}+b_{m+n}}\in f.\]
In particular, if $s_i<b_i$ for some $i$ with $1\le i\le m$, then $g\circ F_1=0$.

\item If $g\circ F_2\neq 0$, then $s_{i}\ge b_{i}$ for any $i=m+1,\dots, m+n$ and 
\[cdx_1^{s_1+b_1}\cdots x_m^{s_m+b_m}x_{m+1}^{s_{m+1}-b_{m+1}} \cdots x_{m+n}^{s_{m+n}-b_{m+n}}\in f.\]
In particular, if $s_i<b_i$ for some $i$ with $m+1\le i\le m+n$, then $g\circ F_2=0$.

\item 
Suppose that $b_j>0$ and $s_j=\max\{t_j\in \mathbb Z_{\ge 0} \mid ex_1^{t_1}\cdots x_{m+n}^{t_{m+n}}\in f, e\neq 0 \}$ for some $j$ with $m+1\le j \le m+n$.
Then $g\circ F_1=0$.
\end{enumerate}
\end{lem}

\begin{proof}

We omit the proof of (1) since (1) can be proved in the same manner as (2).

(2) Since $f\in \operatorname{Ann}_R(F)$, $g\in f$ and $g\circ F_2\neq 0$, there exists a monomial $g'=d'x_1^{s'_1}\cdots x_{m+n}^{s'_{m+n}}$  such that $g'\in f$ and $g\circ cF_2=g'\circ F_1$. 
Note that 
\[
g'\circ F_1=d'x_1^{a_1+b_1-s'_1}\cdots x_m^{a_m+b_m-s'_m}x_{m+1}^{a_{m+1}-s'_{m+1}} \cdots x_{m+n}^{a_{m+n}-s'_{m+n}},
\]
\[
g\circ cF_2=cdx_1^{a_1-s_1}\cdots x_m^{a_m-s_m}x_{m+1}^{a_{m+1}+b_{m+1}-s_{m+1}} \cdots x_{m+n}^{a_{m+n}+b_{m+n}-s_{m+n}}.
\]
By comparing the coefficients and the degrees of $g\circ cF_2$ and $g'\circ F_1$, we have $d'=cd$, $s'_i=s_i+b_i$ for $i=1,\dots, m$ and $s'_j=s_j-b_j$ for $j=m+1,\dots, m+n$.
 Hence $s_j\ge b_j$ for $j=m+1,\dots, m+n$ and $cdx_1^{s_1+b_1}\cdots x_m^{s_m+b_m}x_{m+1}^{s_{m+1}-b_{m+1}} \cdots x_{m+n}^{s_{m+n}-b_{m+n}}\in f$.

(3) Suppose, for the sake of contradiction, that $g\circ F_1\neq 0$.
We have \[c^{-1}dx_1^{s_1-b_1}\cdots x_m^{s_m-b_m}x_{m+1}^{s_{m+1}+b_{m+1}} \cdots x_{m+n}^{s_{m+n}+b_{m+n}}\in f\] by (1).
However, the degree of this monomial in $x_j$ is $s_j+b_j$, which contradicts the assumption that 
\[s_j=\max\{t_j\in \mathbb Z_{\ge 0} \mid ex_1^{t_1}\cdots x_{m+n}^{t_{m+n}}\in f, e\neq 0\}.\]
Therefore, we conclude that $g\circ F_1=0$.

\end{proof}

We will use the following lemma in the proof of Proposition \ref{prop not c.i n>1} and Proposition \ref{prop not c.i n=1}.

\begin{lem}\label{lem chage generator}
Let $(R,\mathfrak m)$ be a local ring and let $f_1,\dots,f_N, h_1,\cdots, h_N, h'_1,\dots,h'_N$ be elements of $R$, where $N\in \mathbb N$ with $N\ge 2$.
Let $g_1,g_2$ be elements of $R$ such that 
\[
g_1=h_1f_1+\cdots+h_Nf_N,\quad g_2=h'_1f_1+\cdots+h'_Nf_N.
\]
Suppose that $h_1,h'_2\not\in \mathfrak m$ and $h'_1\in \mathfrak m$.
Then
\[
(f_1,f_2,f_3,\dots,f_N)=(g_1,g_2,f_3,\dots,f_N)
\]
\end{lem}
\begin{proof}
Let $I=\{1,3,4,\dots,N\}$.
Since $h'_2\not\in \mathfrak m$, we have 
$f_2={h'_2}^{-1}\left(g_2-\sum_{i\in I}h'_if_i\right)$.
Hence, we obtain
\begin{align*}
g_1&=h_1f_1+h_2f_2+\cdots+h_Nf_N\\
&=h_1f_1+h_2{h'_2}^{-1}\left(g_2-\sum_{i\in I}h'_if_i\right)+h_3f_3+\cdots+h_Nf_N\\
&=\sum_{i\in I}\left(h_i-h_2{h'_2}^{-1}h'_i\right)f_i+h_2{h'_2}^{-1}g_2.
\end{align*}
Since $h_1\not\in \mathfrak m$ and $h'_1\in \mathfrak m$,
we conclude that $h_1-h_2{h'_2}^{-1}h'_1\not\in \mathfrak m$.
Therefore, we have 
\[
(g_1,g_2,f_3,\dots,f_N)=(f_1,g_2,f_3,\dots,f_N)=(f_1,f_2,f_3,\dots,f_N)
\]
\end{proof}

\begin{prop}\label{prop not c.i n>1}
Let $m,n\in \mathbb N$ with $m,n\ge 2$.
Let $R=k[x_1,\dots,x_{m+n}]$ and $S=k[X_1,\dots,X_{m+n}]$ be  polynomial rings.
Let \[F=X_1^{a_1}\cdots X_{m+n}^{a_{m+n}}(X_1^{b_1}\cdots X_m^{b_m}-cX_{m+1}^{b_{m+1}}\cdots X_{m+n}^{b_{m+n}})\in S\] 
be a binomial, where $a_1,\dots,a_{m+n},b_1\dots,b_{m+n}\in \mathbb Z_{\ge 0}$ and $c\in k\setminus\{0\}$.
Suppose that 
\[\#\{i \mid b_i\neq 0,\ i=1,\dots, m\}\ge 2,\]
\[\#\{i \mid b_i\neq 0,\ i=m+1,\dots, m+n\}\ge 2.\] 
Then $R/\operatorname{Ann}_R(F)$ is not a complete intersection.
\end{prop}

\begin{proof}
We may assume that $b_1,b_2,b_{m+1},b_{m+2}\ge 1$.
Suppose, for the sake of contradiction, that $R/\operatorname{Ann}_R(F)$ is a complete intersection.
Let $f_1,\dots,f_{m+n}$ be generators of  $\operatorname{Ann}_R(F)$.
Then $f_1,\dots,f_{m+n}$ is a regular sequence on $R$.
We now claim the following.
\begin{claim}\label{claim1}
For $l=1,2$, if $dx_l^{s_l}x_{m+1}^{s_{m+1}}\in f_i$ with $d\in k\setminus\{0\}$, $s_l\le a_l+1$ and $s_{m+1}\le a_{m+1}+1$ for some $i$, then $s_l=a_l+1$ and $s_{m+1}=a_{m+1}+1$.
\end{claim}
\begin{proof}[Proof of Claim \ref{claim1}]
We will prove only the case $l=1$ since the case $l=2$ can be proved in the same way.

Let $g=dx_1^{s_1}x_{m+1}^{s_{m+1}}$.
Since $b_2\ge 1$ and $b_{m+2}\ge 1$, by Lemma \ref{lem property in Ann}(1)(2), we have $g\circ F_1=g \circ F_2=0$, where $F_1=X_1^{a_1+b_1}\cdots X_m^{a_m+b_m}X_{m+1}^{a_{m+1}}\cdots X_{m+n}^{a_{m+n}}$ and
$F_2=X_1^{a_1}\cdots X_m^{a_m}X_{m+1}^{a_{m+1}+b_{m+1}}\cdots X_{m+n}^{a_{m+n}+b_{m+n}}$.
Hence $s_1\ge a_1+1$ and $s_{m+1}\ge a_{m+1}+1$.
This proves the claim.
\end{proof}

We now return to the proof of Proposition \ref{prop not c.i n>1}.
Let $\mathfrak m=(x_1,\dots,x_{m+n})$.
Since $x_1^{a_1+1}x_{m+1}^{a_{m+1}+1}\in \operatorname{Ann}_R(F)$, there exist elements $h_1,\dots, h_{m+n}$ of $R$ such that 
\[
x_1^{a_1+1}x_{m+1}^{a_{m+1}+1}=h_1f_1+\cdots +h_{m+n}f_{m+n}.
\]
By Claim \ref{claim1}, $f_j$ contains the term $x_1^{a_1+1}x_{m+1}^{a_{m+1}+1}$  and $h_j\not\in\mathfrak m$  for some $j$.
After renumbering, we may assume that $f_1$ contains the term $x_1^{a_1+1}x_{m+1}^{a_{m+1}+1}$ and $h_1\not\in\mathfrak m$.
By replacing $f_i$ with $f_i-c_if_1$ for $i=2,\dots, m+n$ and some $c_i\in k$, we may assume that $f_i$ does not contain the terms $x_1^{s_1}x_{m+1}^{s_{m+1}}$ for $s_1\le a_1+1$, $s_{m+1}\le a_{m+1}+1$  and $i=2,\dots, m+n$ by Claim \ref{claim1}.

Since $x_2^{a_2+1}x_{m+1}^{a_{m+1}+1}\in \operatorname{Ann}_R(F)$, there exist elements $h'_1,\dots, h'_{m+n}$ of $R$ such that 
\[
x_2^{a_2+1}x_{m+1}^{a_{m+1}+1}=h'_1f_1+\cdots +h'_{m+n}f_{m+n}.
\]
By Claim \ref{claim1}, we may assume that $f_j$ contains the term $x_2^{a_2+1}x_{m+1}^{a_{m+1}+1}$ and $h'_j\not\in\mathfrak m$ for some $j$.
Since $f_1$ contains the term $x_1^{a_1+1}x_{m+1}^{a_{m+1}+1}$ and $f_i$ does not contain the terms $x_1^{s_1}x_{m+1}^{s_{m+1}}$ for $s_1\le a_1+1$, $s_{m+1}\le a_{m+1}+1$ and $i=2,\dots, m+n$, we have $j\neq 1$ and $h'_1\in\mathfrak m$.
Hence we may assume that $f_2$ contains the term $x_2^{a_2+1}x_{m+1}^{a_{m+1}+1}$ and $h'_2\not\in\mathfrak m$.
Therefore we have
\[
(f_1,\dots,f_{m+n})R_{\mathfrak m}=(x_1^{a_1+1}x_{m+1}^{a_{m+1}+1}, x_2^{a_2+1}x_{m+1}^{a_{m+1}+1}, f_3,\dots,f_{m+n})R_{\mathfrak m}
\]
by Lemma \ref{lem chage generator}.
Since $x_1^{a_1+1}x_{m+1}^{a_{m+1}+1}, x_2^{a_2+1}x_{m+1}^{a_{m+1}+1}$ is not a regular sequence,
this contradicts $f_1,\dots,f_{m+n}$ is a regular sequence on $R$.
Therefore $R/\operatorname{Ann}_R(F)$ is not a complete intersection.
\end{proof}

\begin{prop}\label{prop not c.i n=1}
Let $m\in \mathbb N$ with $m\ge 2$.
Let $R=k[x_1,\dots,x_{m+1}]$ and $S=k[X_1,\dots,X_{m+1}]$ be  polynomial rings.
Let \[F=X_1^{a_1}\cdots X_{m+1}^{a_{m+1}}(X_1^{b_1}\cdots X_m^{b_m}-cX_{m+1}^{b_{m+1}})\in S\] 
be a binomial, where $a_1,\dots,a_{m+1},b_1\dots,b_{m+1}\in \mathbb Z_{\ge 0}$ with $b_{m+1}\ge 1$ and $c\in k\setminus\{0\}$.
Suppose that \[
\#\{i \mid b_i\neq 0,\ i=1,\dots, m\}\ge 2.
\]
Let \[v=\min\{i\in \mathbb N \mid (x_1^{b_1}\cdots x_m^{b_m})^i\circ (X_1^{a_1}\cdots X_m^{a_m})=0\}.\]
Suppose $a_{m+1}+1< vb_{m+1}$.
Then $R/\operatorname{Ann}_R(F)$ is not a complete intersection.
\end{prop}

\begin{proof}
We may assume that $b_1,b_2\ge 1$.
Let \[w=\max\{i\in \mathbb Z_{\ge 0} \mid a_{m+1}+1\ge ib_{m+1}\}.\]
Since $wb_{m+1}\le a_{m+1}+1<vb_{m+1}$, we have $w<v$.
Hence we have $(x_1^{b_1}\cdots x_m^{b_m})^w\circ (X_1^{a_1}\cdots X_m^{a_m})\neq 0$, which implies that 
\[a_i\ge wb_i\ \text{ for } i=1,\dots, m.\]

Suppose, for the sake of contradiction, that $R/\operatorname{Ann}_R(F)$ is a complete intersection.
Let $f_1,\dots,f_{m+1}$ be generators of  $\operatorname{Ann}_R(F)$.
Then $f_1,\dots,f_{m+1}$ is a regular sequence on $R$.
We now claim the following.
\begin{claim}\label{claim2}
For $l=1,2$,
if $dx_l^{s_l}x_{m+1}^{s_{m+1}}\in f_i$ with $d\in k\setminus\{0\}$, $s_l\le a_l+1-wb_l$ and $s_{m+1}\le a_{m+1}+1$ for some $i$, then $s_l=a_l+1-wb_l$ and $s_{m+1}=a_{m+1}+1$.
\end{claim}
\begin{proof}[Proof of Claim \ref{claim2}]
We will prove only the case $l=1$ since the case $l=2$ can be proved in the same way.

Let $g_0=dx_1^{s_1}x_{m+1}^{s_{m+1}}$.
Since $b_2\ge 1$, by Lemma \ref{lem property in Ann}(1), we have $g_0\circ F_1=0$.
Thus, $s_1\ge a_1+b_1+1$ or $s_{m+1}\ge a_{m+1}+1$.
Since $b_1\ge 1$ and $s_1\le a_1+1-wb_1$, it follows that $s_{m+1}\ge a_{m+1}+1$.
Therefore, we have $s_{m+1}= a_{m+1}+1$.

Next, we will prove that $s_1= a_1+1-wb_1$.
Since $s_1\le a_1+1-wb_1$, it is enough to prove that $s_1\ge a_1+1-wb_1$.
If $w=0$, then $s_{m+1}<b_{m+1}$ by the definition of $w$.
By Lemma \ref{lem property in Ann}(2), we have $g_0\circ F_2=0$.
Therefore we have $s_1\ge a_1+1$.
This implies that $s_1\ge a_1+1-wb_1$ if $w=0$.

Suppose that $w>0$.
Let $g_j:=c^jx_1^{s_1+jb_1}x_2^{jb_2}\cdots x_m^{jb_m}x_{m+1}^{s_{m+1}-jb_{m+1}}$ for $j=1,\dots, w$.
Since $s_1\le a_1+1-wb_1$, $a_i\ge wb_i$ for  $i=1,\dots, m$ and $s_{m+1}=a_{m+1}+1\ge wb_{m+1}$,
we obtain $g_j\circ F_2\neq 0$ for $j=0,\dots, w-1$.
By applying Lemma \ref{lem property in Ann}(2) repeatedly to $g_j$ for $j=0,\dots,w-1$, we have $g_1,\dots,g_w\in f_i$.
Since $s_{m+1}-wb_{m+1}=a_{m+1}+1-wb_{m+1}<b_{m+1}$ by the definition of $w$, we have $g_w\circ F_2=0$ by Lemma \ref{lem property in Ann}(2).
Thus, we have $s_1+wb_1\ge a_1+1$, which implies that $s_1\ge a_1+1-wb_1$.
\end{proof}

We now return to the proof of Proposition \ref{prop not c.i n=1}.
Let \[q=\sum_{i=0}^{w}c^i(x_1^{b_1}\cdots x_m^{b_m})^ix_{m+1}^{a_{m+1}+1-ib_{m+1}},\]
\[F_1=X_1^{a_1+b_1}\cdots X_m^{a_m+b_m}X_{m+1}^{a_{m+1}} \text{  and  } F_2=X_1^{a_1}\cdots X_m^{a_m}X_{m+1}^{a_{m+1}+b_{m+1}}.\]
Since 
\[(x_1^{b_1}\cdots x_m^{b_m})^{i+1}x_{m+1}^{a_{m+1}+1-(i+1)b_{m+1}}\circ F_1=(x_1^{b_1}\cdots x_m^{b_m})^ix_{m+1}^{a_{m+1}+1-ib_{m+1}}\circ F_2\] 
for any $i=0,\dots, w-1$ and $x_{m+1}^{a_{m+1}+1}\circ F_1=x_1^{a_1+1}\circ F_2=x_2^{a_2+1}\circ F_2=0$,
we have \[x_1^{a_1+1-wb_1}q,\, x_2^{a_2+1-wb_2}q\in\operatorname{Ann}_R(F).\]
Let $\mathfrak m=(x_1,\dots,x_{m+1})$.
Since $x_1^{a_1+1-wb_1}q\in \operatorname{Ann}_R(F)$, there exist elements $h_1,\dots, h_{m+1}$ of $R$ such that 
\[
x_1^{a_1+1-wb_1}q=h_1f_1+\cdots +h_{m+1}f_{m+1}.
\]

Note that $x_1^{a_1+1-wb_1}x_{m+1}^{a_{m+1}+1}\in x_1^{a_1+1-wb_1}q$.
By Claim \ref{claim2}, $f_j$ contains the term $x_1^{a_1+1-wb_1}x_{m+1}^{a_{m+1}+1}$  and $h_j\not\in\mathfrak m$  for some $j$.
After renumbering, we may assume that $f_1$ contains the term $x_1^{a_1+1-wb_1}x_{m+1}^{a_{m+1}+1}$ and $h_1\not\in\mathfrak m$.
By replacing $f_i$ with $f_i-c_if_1$ for $i=2,\dots, m+1$ and some $c_i\in k$, we may assume that $f_i$ does not contain the terms $x_1^{s_1}x_{m+1}^{s_{m+1}}$ for $s_1\le a_1+1-wb_1$, $s_{m+1}\le a_{m+1}+1$,  and $i=2,\dots, m+1$ by Claim \ref{claim2}.

Since $x_2^{a_2+1-wb_2}q\in \operatorname{Ann}_R(F)$, there exist elements $h'_1,\dots, h'_{m+n}$ of $R$ such that 
\[
x_2^{a_2+1-wb_2}q=h'_1f_1+\cdots +h'_{m+1}f_{m+1}.
\]
Note that $x_2^{a_2+1-wb_1}x_{m+1}^{a_{m+1}+1}\in x_2^{a_2+1-wb_1}q$.
By Claim \ref{claim2}, we may assume that $f_j$ contains the term $x_2^{a_2+1-wb_2}x_{m+1}^{a_{m+1}+1}$ and $h'_j\not\in\mathfrak m$ for some $j$.
Since $f_1$ contains the term $x_1^{a_1+1-wb_1}x_{m+1}^{a_{m+1}+1}$ and $f_i$ does not contain the terms $x_1^{s_1}x_{m+1}^{s_{m+1}}$ for $s_1\le a_1+1-wb_1$, $s_{m+1}\le a_{m+1}+1$ and $i=2,\dots, m+1$, we have $j\neq 1$ and $h'_1\in\mathfrak m$.
Hence we may assume that $f_2$ contains the term $x_2^{a_2+1-wb_2}x_{m+1}^{a_{m+1}+1}$ and $h'_2\not\in\mathfrak m$.
Therefore we have
\[
(f_1,\dots,f_{m+1})R_{\mathfrak m}=(x_1^{a_1+1-wb_1}q, x_2^{a_2+1-wb_2}q, f_3,\dots,f_{m+1})R_{\mathfrak m}
\]
by Lemma \ref{lem chage generator}.
Since $x_1^{a_1+1-wb_1}q, x_2^{a_2+1-wb_2}q$ is not a regular sequence,
this contradicts $f_1,\dots,f_{m+1}$ is a regular sequence on $R$.
Therefore $R/\operatorname{Ann}_R(F)$ is not a complete intersection.
\end{proof}

\begin{prop}\label{prop c.i n=1}
Let $m\in \mathbb N$.
Let $R=k[x_1,\dots,x_{m+1}]$ and $S=k[X_1,\dots,X_{m+1}]$ be  polynomial rings.
Let \[F=X_1^{a_1}\cdots X_{m+1}^{a_{m+1}}(X_1^{b_1}\cdots X_m^{b_m}-cX_{m+1}^{b_{m+1}})\in S\] 
be a binomial, where $a_1,\dots,a_{m+1},b_1\dots,b_{m+1}\in \mathbb Z_{\ge 0}$ with $b_{m+1}\ge 1$ and $c\in k\setminus\{0\}$.
Suppose that \[
\#\{i \mid b_i\neq 0,\ i=1,\dots, m\}\ge 1.
\]
Let \[v=\min\{i\in \mathbb N \mid (x_1^{b_1}\cdots x_m^{b_m})^i\circ (X_1^{a_1}\cdots X_m^{a_m})=0\}.\]
Suppose $a_{m+1}+1\ge vb_{m+1}$.
Let \[
p=\sum_{i=0}^{v}c^i(x_1^{b_1}\cdots x_m^{b_m})^ix_{m+1}^{a_{m+1}+1-ib_{m+1}}\in R
\]
and
\[
I=(x_1^{a_1+b_1+1}, \dots, x_m^{a_m+b_m+1}, p)\subset R.
\]
Then
\begin{enumerate}
\item $x_i^{a_i+1}x_{m+1}^{a_{m+1}+1},x_{m+1}^{a_{m+1}+b_{m+1}+1}\in I$ for any $i=1,\dots, m$.

\item
$I=\operatorname{Ann}_R(F)$.
In particular, $R/\operatorname{Ann}_R(F)$ is a complete intersection.
\end{enumerate}

\end{prop}

\begin{proof}
We may assume that $b_1\ge 1$.

First, we prove (1).  
Since 
$x_j^{a_j+1}(x_1^{b_1}\cdots x_m^{b_m})\in I$ and
\[
x_j^{a_j+1}x_{m+1}^{a_{m+1}+1}=x_j^{a_j+1}p-x_j^{a_j+1}\sum_{i=1}^{v}c^i(x_1^{b_1}\cdots x_m^{b_m})^ix_{m+1}^{a_{m+1}+1-ib_{m+1}}
\]
for any $j=1,\dots,m$,
we have $x_j^{a_j+1}x_{m+1}^{a_{m+1}+1}\in I$ for any $j=1,\dots,m$.
 
By the definition of $v$, there exists $i$ such that $1\le i \le m$ and $a_i+1\le vb_i$.
Therefore we have $(x_1^{b_1}\cdots x_m^{b_m})^{v+1}\in I$.
Since 
$(x_1^{b_1}\cdots x_m^{b_m})^{v+1}\in I$ and
\[
x_{m+1}^{a_{m+1}+b_{m+1}+1}=x_{m+1}^{b_{m+1}}p-cx_1^{b_1}\cdots x_m^{b_m}p+c^{v+1}(x_1^{b_1}\cdots x_m^{b_m})^{v+1}x_{m+1}^{a_{m+1}+1-vb_{m+1}},
\]
we have $x_{m+1}^{a_{m+1}+b_{m+1}+1}\in I$.

Next, we prove (2).
Let \[F_1=X_1^{a_1+b_1}\cdots X_m^{a_m+b_m}X_{m+1}^{a_{m+1}} \text{  and  } F_2=X_1^{a_1}\cdots X_m^{a_m}X_{m+1}^{a_{m+1}+b_{m+1}}.\]
Since 
\[(x_1^{b_1}\cdots x_m^{b_m})^{i+1}x_{m+1}^{a_{m+1}+1-(i+1)b_{m+1}}\circ F_1=(x_1^{b_1}\cdots x_m^{b_m})^ix_{m+1}^{a_{m+1}+1-ib_{m+1}}\circ F_2\] 
for any $i=0,\dots, v-1$ and 
\[x_{m+1}^{a_{m+1}+1}\circ F_1=(x_1^{b_1}\cdots x_m^{b_m})^vx_{m+1}^{a_{m+1}+1-vb_{m+1}}\circ F_2=0,\]
we have \[p\in\operatorname{Ann}_R(F).\]
Since $x_i^{a_i+b_i+1}, p\in \operatorname{Ann}_R(F)$ for any $i=1,\dots, m$, we have $I\subset \operatorname{Ann}_R(F)$.
Hence it is enough to show that $I\supset \operatorname{Ann}_R(F)$.
Suppose, for the sake of contradiction, that $I\not\supset \operatorname{Ann}_R(F)$.

For $f=\sum_{(i_1,\dots,i_{m+1})\in \mathbb Z_{\ge 0}^{m+1}}c_{i_1,\dots,i_{m+1}}x_1^{i_1}\cdots x_{m+1}^{i_{m+1}}\in R$,
we define \[N(f):=\#\{(i_1,\dots,i_{m+1})\in\mathbb Z_{\ge 0}^{m+1}  \mid c_{i_1,\dots,i_{m+1}}\neq 0\}.\]
Let $f$ be an element of $\operatorname{Ann}_R(F)\setminus I$ such that $N(f)$ is minimized over all elements of $\operatorname{Ann}_R(F)\setminus I$.
If $s_i\ge a_i+b_i+1$ for some $i$ with $1\le i \le m+1$, then  $x_1^{s_1}\cdots x_{m+1}^{s_{m+1}}\in I$ by (1) and $f+ex_1^{s_1}\cdots x_{m+1}^{s_{m+1}}\in\operatorname{Ann}_R(F)\setminus I$ for any $e\in k$.
Hence, if $dx_1^{s_1}\cdots x_{m+1}^{s_{m+1}}\in f$ with $d\in k\setminus\{0\}$ and $s_i\in \mathbb Z_{\ge 0}$, we have $s_1\le a_1+b_1,\dots, s_{m+1}\le a_{m+1}+b_{m+1}$.
Let
\[t_{m+1}=\max\{s_{m+1}\in \mathbb Z_{\ge 0} \mid dx_1^{s_1}\cdots x_{m+1}^{s_{m+1}}\in f, d\in k\setminus\{0\}\}\]
and let $g_0=ex_1^{t_1}\cdots x_m^{t_m}x_{m+1}^{t_{m+1}}$ be a monomial with $g_0\in f$.
Then $t_i\le a_{i}+b_{i}$ for any $i=1,\dots, m+1$.
We have $g_0\circ F_1=0$ by Lemma \ref{lem property in Ann}(3).
Hence $t_{m+1}\ge a_{m+1}+1$.
Note that $t_{m+1}\ge a_{m+1}+1\ge vb_{m+1}$.
Let $g_j:=c^jex_1^{t_1+jb_1}\cdots x_m^{t_m+jb_m}x_{m+1}^{t_{m+1}-jb_{m+1}}$ for $j=1,\dots, v$.
Let $w=\min\{i\in \mathbb Z_{\ge 0} \mid g_i\circ (x_1^{a_1}\cdots x_m^{a_m})=0\}.$
Then $w \le v$, $g_{w}\circ F_2=0$ and $g_i\circ F_2\neq 0$ for $i=0,\dots, w-1$.
Moreover we have $t_l+wb_l\ge a_l+1$ for some $l$ with $1\le l \le m$.
Hence $g_j\in (x_l^{a_l+b_l+1})\subset I$ for $j=w+1,\dots v$. 
Since  \[
\sum_{i=0}^{w}g_i =ex_1^{t_1}\cdots x_m^{t_m}x_{m+1}^{t_{m+1}-a_{m+1}-1}p-\sum_{i=w+1}^{v}g_i,
\]
we have $\sum_{i=0}^{w}g_i\in I$.
Since $g_0\in f$ and $g_i\circ F_2\neq 0$ for $i=0,\dots, w-1$, we have $g_1,\dots, g_{w}\in f$ by Lemma \ref{lem property in Ann}(2).
Hence $f-\sum_{i=0}^{w}g_i\in  \operatorname{Ann}_R(F)\setminus I$ and $g_0,\dots,g_w\not\in f-\sum_{i=0}^{w}g_i$.
Thus $N(f-\sum_{i=0}^{w}g_i)=N(f)-w-1$, which contradicts $f$ is an element of $\operatorname{Ann}_R(F)\setminus I$ such that $N(f)$ is minimized over all elements of $\operatorname{Ann}_R(F)\setminus I$.
Hence we have $I\supset  \operatorname{Ann}_R(F)$.
\end{proof}

The following lemma is used in Lemma \ref{lem c.i m=n=1 Ann(F)} to determine $\operatorname{Ann}_R(F)$.

\begin{lem}\label{lem equal of c.i ideal}
Let $R=k[x_1,\dots, x_N]$ and $\mathfrak m=(x_1,\dots, x_N)$.
Let $(f_1,\dots,f_N)$ and $(g_1\cdots, g_N)$ be $\mathfrak m$-primary ideals.
Let $A=(a_{ij})\in M_N(R)$ be an $N\times N$ matrix such that $(g_1,\dots,g_N)=(x_1,\dots,x_N)A$, that is $g_j=\sum_{i=1}^Nx_ia_{ij}$ for $j=1,\dots, N$.
Suppose that 
$(g_1\cdots, g_N)\subset (f_1,\dots,f_N)$ and $\operatorname{det} A\not \in (f_1,\dots,f_N)$.
Then $(g_1\cdots, g_N)= (f_1,\dots,f_N)$.
\end{lem}

\begin{proof}
We put $I=(f_1,\dots,f_N)$ and $J=(g_1,\dots,g_N)$. 
Let $B,C\in M_N(R)$ be $N\times N$ matrices such that $(f_1,\dots,f_N)=(x_1,\dots,x_N)B$ and $(g_1,\dots,g_N)=(f_1,\dots,f_N)C$.
Then  we have $(g_1,\dots,g_N)=(x_1,\dots,x_N)A=(x_1,\dots,x_N)(BC)$.
By \cite[Corollary 2.3.10]{BH}, we have
\[J:\mathfrak m=J+(\operatorname{det} A)=J+(\operatorname{det} (BC)),\
I:\mathfrak m=I+(\operatorname{det} B),\
J:I=J+(\operatorname{det} C).
\]
Hence we obtain
$\operatorname{det} A\in J+(\operatorname{det} (BC))$, $\operatorname{det} B\mathfrak m\subset I$ and 
$\operatorname{det} CI\subset J$.
Suppose, for the sake of contradiction, that $J\neq I$.
This implies that $JR_{\mathfrak m} \neq IR_{\mathfrak m}$.
Indeed, since $I$ and  $J$ are $\mathfrak m$-primary ideals, we have $JR_{\mathfrak n} = IR_{\mathfrak n}$ for any maximal ideal of $R$ with $\mathfrak n\neq \mathfrak m$.
Therefore, we obtain $JR_{\mathfrak m} \neq IR_{\mathfrak m}$ by \cite[Corollary 2.9]{Eis}.
Since 
$\operatorname{det} CIR_{\mathfrak m}\subset JR_{\mathfrak m},$
it follows that $\operatorname{det} C\in\mathfrak m$.
Therefore we have
 \[\operatorname{det} A\in J+(\operatorname{det} (BC))= J+(\operatorname{det} B) (\operatorname{det} C)\subset J+\operatorname{det} B\mathfrak m\subset I,\]
which contradicts the assumption that $\operatorname{det} A\not \in I$.
Hence we conclude that $J=I$.
\end{proof}

\begin{lem}\label{lem c.i m=n=1 Ann(F)}
Let $R=k[x_1,x_2]$ and $S=k[X_1,X_2]$ be  polynomial rings.
Let \[F=X_1^{a_1}X_{2}^{a_{2}}(c_1X_1^{b_1}-c_2X_{2}^{b_{2}})\in S\] 
be a binomial, where $a_1,a_{2},b_1,b_{2}\in \mathbb Z_{\ge 0}$ with $b_1, b_{2}\ge 1$ and $c_1,c_2\in k\setminus\{0\}$.
For $j=1,2$, let \[v_j=\min\{i\in \mathbb N \mid a_j+1\le ib_j\}.\]
Then
\begin{enumerate}
\item If $v_1<v_2$, then $\operatorname{Ann}_R(F)=\left(x_1^{a_1+b_1+1},\sum_{i=0}^{v_1}c_1^{v_1-i}c_2^ix_1^{ib_1}x_2^{a_2+1-ib_2}\right)$.

\item If $v_1>v_2$, then $\operatorname{Ann}_R(F)=\left(x_2^{a_2+b_2+1},\sum_{i=0}^{v_2}c_1^ic_2^{v_2-i}x_1^{a_1+1-ib_1}x_2^{ib_2}\right)$.

\item If $v_1=v_2$, then 
\[\operatorname{Ann}_R(F)=\left(\sum_{i=0}^{v_1}c_1^{v_1-i}c_2^ix_1^{ib_1}x_2^{(v_1-i)b_2},\sum_{i=0}^{v_1-1}c_1^{v_1-1-i}c_2^ix_1^{a_1+1-(v_1-1-i)b_1}x_2^{a_2+1-ib_2}\right).\]
\end{enumerate}

\end{lem}

\begin{proof}
(1)
Note that $v_j=\min\{i\in \mathbb N \mid x_j^{ib_j}\circ X_j^{a_j}=0\}$ for $j=1,2$.
Since $v_1<v_2$, then $v_1b_2<a_2+1$ by the definition of $v_2$.
Hence we have
\begin{align*}
\operatorname{Ann}_R(F)&=\operatorname{Ann}_R(c_1^{-1}F)\\
&=\left(x_1^{a_1+b_1+1},\sum_{i=0}^{v_1}\left(\frac{c_2}{c_1}\right)^ix_1^{ib_1}x_2^{a_2+1-ib_2}\right)\\
&=\left(x_1^{a_1+b_1+1},\sum_{i=0}^{v_1}c_1^{v_1-i}c_2^ix_1^{ib_1}x_2^{a_2+1-ib_2}\right)
\end{align*}
by Proposition \ref{prop c.i n=1}(2).

(2) We omit the proof of (2) since (2) can be proved in the same manner as (1).

(3)
Since any Gorenstein local ring of embedding dimension at most two is a complete intersection by \cite[Corollary 21.20]{Eis}, $R/\operatorname{Ann}_{R}(F)$ is a complete intersection.
Hence $\operatorname{Ann}_{R}(F)$ is generated by two elements of $R$.

We put $v=v_1$.
Let \[p=\sum_{i=0}^{v}c_1^{v-i}c_2^ix_1^{ib_1}x_2^{(v-i)b_2},\quad q=\sum_{i=0}^{v-1}c_1^{v-1-i}c_2^ix_1^{a_1+1-(v-1-i)b_1}x_2^{a_2+1-ib_2}.\]
Let $F_1=X_1^{a_1+b_1}X_{2}^{a_{2}}$  and  $F_2=X_1^{a_1}X_{2}^{a_{2}+b_{2}}$.
Then $F=c_1F_1-c_2F_2$.
Since 
\[x_1^{(i+1)b_1}x_{2}^{(v-1-i)b_{2}}\circ F_1=x_1^{ib_1}x_{2}^{(v-i)b_{2}}\circ F_2,\]
 \[x_1^{a_1+1-(v-2-j)b_1}x_2^{a_2+1-(j+1)b_2}\circ F_1=x_1^{a_1+1-(v-1-j)b_1}x_2^{a_2+1-jb_2}\circ F_2\] 
for any $i=0,\dots, v-1$, $j=0,\cdots, v-2$ and 
\[x_{2}^{v}\circ F_1=x_1^v\circ F_2=x_{2}^{a_2+1}\circ F_1=x_1^{a_1+1}\circ F_2=0,\]
we have $p,q\in\operatorname{Ann}_R(F).$
Therefore $(p,q)\subset\operatorname{Ann}_R(F).$

Since \[c_2^{v}x_1^{(v+1)b_1}=x_1^{b_1}p-c_1x_1^{vb_1-a_1-1}x_2^{vb_2-a_2-1}q,\] 
\[c_1^vx_2^{(v+1)b_2}=x_2^{b_2}p-c_2x_1^{vb_1-a_1-1}x_2^{vb_2-a_2-1}q,\]
it follows that $(p,q)$ is a $(x_1,x_2)$-primary ideal.
Note that $q\in (x_2)$ because $a_2+1-(v_2-1)b_2>0$ by the definition of $v_2$, and $v=v_1=v_2$.
We have
\[
(p,q)=(x_1,x_2)\begin{pmatrix}
   c_2^{v}x_1^{vb_1-1} & 0 \\
   \dfrac{p-c_2^{v}x_1^{vb_1}}{x_2} & \dfrac{q}{x_2}
\end{pmatrix}
\]

The above $2\times 2$ matrix is denoted by $A$.
Then 
\[\operatorname{det} A=c_2^{v}x_1^{vb_1-1}\dfrac{q}{x_2}=\sum_{i=0}^{v-1}c_1^{v-1-i}c_2^{v+i}x_1^{a_1+(i+1)b_1}x_2^{a_2-ib_2}.\]
Since $x_1^{a_1+(i+1)b_1}\circ F=0$ for $i=1,\dots, v-1$,
we have 
\[
\operatorname{det} A\circ F=c_1^{v-1}c_2^{v}x_1^{a_1+b_1}x_2^{a_2}\circ F=c_1^{v-1}c_2^{v}\neq 0.
\]
Hence we have $\operatorname{det} A \not \in \operatorname{Ann}_R(F)$.
Since $(p,q)\subset \operatorname{Ann}_R(F)$,
 we conclude that
\[\operatorname{Ann}_R(F)=(p,q)\]
by Lemma \ref{lem equal of c.i ideal},
\end{proof}

\begin{lem}\label{lem c.i add variable}
Let $m,n\in \mathbb N$.
Let $R'=k[x_1,\dots,x_m]$, $S'=k[X_1,\dots,X_m]$, $R=k[x_1,\dots,x_{m+n}]$ and $S=k[X_1,\dots,X_{m+n}]$ be  polynomial rings.
Let \[G\in S',\ F=X_{m+1}^{a_{m+1}}\cdots X_{m+n}^{a_{m+n}}G\in S,\] where $a_{m+1},\dots,a_{m+n}\in \mathbb Z_{\ge 0}$.
Then  \[\operatorname{Ann}_R(F)=\operatorname{Ann}_{R'}(G)R+(x_{m+1}^{a_{m+1}+1},\dots,a_{m+n}^{a_{m+n}+1}).\]
In particular, $R'/\operatorname{Ann}_{R'}(G)$ is a complete intersection if and only if $R/\operatorname{Ann}_{R}(F)$ is a complete intersection.
\end{lem}

\begin{proof}
Let $I=\operatorname{Ann}_{R'}(G)R+(x_{m+1}^{a_{m+1}+1},\dots,a_{m+n}^{a_{m+n}+1})$.
Since $I\subset \operatorname{Ann}_R(F)$, it is enough to show that $I\supset \operatorname{Ann}_R(F)$.
Let $f\in\operatorname{Ann}_R(F)$, and write \[f=\sum_{(i_{m+1},\dots,i_{m+n})\in \mathbb Z_{\ge 0}^n}f_{i_{m+1},\dots,i_{m+n}}x_{m+1}^{i_{m+1}}\cdots x_{m+n}^{i_{m+n}},\] where $f_{i_{m+1},\dots,i_{m+n}}\in R'$.

If $(i_{m+1},\dots, i_{m+n})\neq (i'_{m+1},\dots, i'_{m+n})$ with $i_j,i'_j \le a_j$ for all $j=m+1,\dots, m+n$ and 
\[f_{i_{m+1},\dots,i_{m+n}}x_{m+1}^{i_{m+1}}\cdots x_{m+n}^{i_{m+n}}\circ F\neq 0, \quad f_{i'_{m+1},\dots,i'_{m+n}}x_{m+1}^{i'_{m+1}}\cdots x_{m+n}^{i'_{m+n}}\circ F\neq 0,\]
 then \[f_{i_{m+1},\dots,i_{m+n}}x_{m+1}^{i_{m+1}}\cdots x_{m+n}^{i_{m+n}}\circ F\neq  f_{i'_{m+1},\dots,i'_{m+n}}x_{m+1}^{i'_{m+1}}\cdots x_{m+n}^{i'_{m+n}}\circ F\]
by comparing their degrees.
Therefore, if $i_j\le a_j$ for all $j=m+1,\dots, m+n$, then \[f_{i_{m+1},\dots,i_{m+n}}x_{m+1}^{i_{m+1}}\cdots x_{m+n}^{i_{m+n}}\circ F=0,\] which implies that $f_{i_{m+1},\dots,i_{m+n}}\circ F=0$.
Thus $f_{i_{m+1},\dots,i_{m+n}}\in \operatorname{Ann}_{R'}(G)$ if $i_j\le a_j$ for all $j=m+1,\dots, m+n$.
This concludes that $I\supset \operatorname{Ann}_R(F)$.
\end{proof}

The following is the main theorem of this paper.

\begin{thm}\label{main thm c.i}
Let $m,n\in \mathbb N$.
Let $R=k[x_1,\dots,x_{m+n}]$ and $S=k[X_1,\dots,X_{m+n}]$ be  polynomial rings.
Let \[F=X_1^{a_1}\cdots X_{m+n}^{a_{m+n}}(c_1X_1^{b_1}\cdots X_m^{b_m}-c_2X_{m+1}^{b_{m+1}}\cdots X_{m+n}^{b_{m+n}} )\in S\] 
be a binomial, where $a_1,\dots,a_{m+n},b_1\dots,b_{m+n}\in \mathbb Z_{\ge 0}$ and $c_1,c_2\in k\setminus \{0\}$.
Let 
\[d_1=\#\{i \mid b_i\neq 0,\ i=1,\dots, m\},\quad d_2=\#\{i \mid b_i\neq 0,\ i=m+1,\dots, m+n\}.\] 
Suppose that $d_1\ge d_2\ge 1$, \[b_{d_1+1}=b_{d_1+2}=\cdots=b_{m}=0,\quad b_{m+d_2+1}=b_{m+d_2+2}=\cdots=b_{m+n}=0.\]
Let \[v=\min\{i\in \mathbb N \mid (x_1^{b_1}\cdots x_m^{b_m})^i\circ (X_1^{a_1}\cdots X_m^{a_m})=0\}.\]
Then 
\begin{enumerate}
\item $R/\operatorname{Ann}_R(F)$ is a complete intersection if and only if one of the following conditions holds:

\begin{enumerate}
\item $d_1=d_2=1$.
\item $d_1\ge 2$, $d_2=1$ and $a_{m+1}+1\ge vb_{m+1}$.
\end{enumerate}

\item Suppose that $d_1=d_2=1$. Let $w=\min\{i\in \mathbb N \mid a_{m+1}+1\le ib_{m+1}\}$ and $I=(x_2^{a_2+1},\cdots, x_m^{a_m+1},x_{m+2}^{a_{m+2}+1},\cdots, x_{m+n}^{a_{m+n}+1})$.
\begin{enumerate}
\item If $v<w$, then 
\[\operatorname{Ann}_R(F)=(x_1^{a_1+b_1+1},\sum_{i=0}^{v}c_1^{v-i}c_2^ix_1^{ib_1}x_{m+1}^{a_{m+1}+1-ib_{m+1}})+I.\]

\item If $v>w$, then 
\[\operatorname{Ann}_R(F)=(x_{m+1}^{a_{m+1}+b_{m+1}+1},\sum_{i=0}^{w}c_1^{i}c_2^{w-i}x_1^{a_1+1-ib_1}x_{m+1}^{ib_{m+1}})+I.\]

\item If $v=w$, then $\operatorname{Ann}_R(F)=(p,q)+I,$
where \[p=\sum_{i=0}^{v}c_1^{v-i}c_2^ix_1^{ib_1}x_{m+1}^{(v-i)b_{m+1}},\quad q=\sum_{i=0}^{v-1}c_1^{v-1-i}c_2^ix_1^{a_1+1-(v-1-i)b_1}x_{m+1}^{a_{m+1}+1-ib_{m+1}}.\]
\end{enumerate}

\item Suppose that $d_2=1$ and $a_{m+1}+1\ge vb_{m+1}$. 
Then
\[
\operatorname{Ann}_R(F)=(x_1^{a_1+b_1+1}, \dots, x_m^{a_m+b_m+1}, p,x_{m+2}^{a_{m+2}+1},\cdots, x_{m+n}^{a_{m+n}+1}),
\]
where $p=\sum_{i=0}^{v}c_1^{v-i}c_2^i(x_1^{b_1}\cdots x_m^{b_m})^ix_{m+1}^{a_{m+1}+1-ib_{m+1}}$.
\end{enumerate}

\end{thm}

\begin{proof}
Note that $b_1>0$ and $b_{m+1}>0$ by the assuption that $d_1\ge d_2\ge 1$, \[b_{d_1+1}=b_{d_1+2}=\cdots=b_{m}=0,\quad b_{m+d_2+1}=b_{m+d_2+2}=\cdots=b_{m+n}=0.\]

If $d_2\ge 2$, then $R/\operatorname{Ann}_R(F)=R/\operatorname{Ann}_R(c_1^{-1}F)$ is not a complete intersection by applying Proposition \ref{prop not c.i n>1} to $c_1^{-1}F$.

If $d_1=d_2=1$, then  
$
F=X_1^{a_1}\cdots X_{m+n}^{a_{m+n}}(c_1X_1^{b_1}-c_2X_{m+1}^{b_{m+1}})
$
and
$v=\min\{i\in \mathbb N \mid a_1+1\le ib_1\}$.
Let $F'=X_1^{a_1}X_{m+1}^{a_{m+1}}(c_1X_1^{b_1}-c_2X_{m+1}^{b_{m+1}}).$
Then we have \[F=X_{2}^{a_{2}}\cdots X_{m}^{a_{m}}X_{m+2}^{a_{m+2}}\cdots X_{m+n}^{a_{m+n}}F'.\]
Therefore $R/\operatorname{Ann}_R(F)$ is a complete intersection and (2) holds  by applying Lemma \ref{lem c.i m=n=1 Ann(F)} to $F'$ and Lemma \ref{lem c.i add variable}. 

We assume that $d_1\ge 2$ and $d_2=1$.
Since $d_2=1$, we have 
\[b_{m+2}=b_{m+3}=\cdots=b_{m+n}=0.\]
Thus,
$
F=X_1^{a_1}\cdots X_{m+n}^{a_{m+n}}(c_1X_1^{b_1}\cdots X_{m}^{b_{m}}-c_2X_{m+1}^{b_{m+1}}).
$
Let $R'=k[x_1,\dots,x_{m+1}]$, $S'=k[X_1,\dots,X_{m+1}]$ be  polynomial rings and \[G=X_1^{a_1}\cdots X_{m+1}^{a_{m+1}}(c_1X_1^{b_1}\cdots X_{m}^{b_{m}}-c_2X_{m+1}^{b_{m+1}})\in S'.\]
Then $F=X_{m+2}^{a_{m+2}}\cdots X_{m+n}^{a_{m+n}}G$.
Note that 
$\operatorname{Ann}_R(F)=\operatorname{Ann}_R(c_1^{-1}F)$ and $\operatorname{Ann}_{R'}(G)=\operatorname{Ann}_{R'}(c_1^{-1}G)$.
By applying Proposition \ref{prop not c.i n=1} and Proposition \ref{prop c.i n=1}(2) to $c_1^{-1}G$, and Lemma \ref{lem c.i add variable}, we conclude that $R/\operatorname{Ann}_R(F)$ is a complete intersection if and only if $R'/\operatorname{Ann}_{R'}(G)$ is a complete intersection, which is equivalent to $a_{m+1}+1\ge vb_{m+1}$.
This completes the proof of (1).

To show (3), assume that $d_2=1$ and $a_{m+1}+1\ge vb_{m+1}$. 
Let \[p'=\sum_{i=0}^{v}\left(\frac{c_2}{c_1}\right)^i(x_1^{b_1}\cdots x_m^{b_m})^ix_{m+1}^{a_{m+1}+1-ib_{m+1}}\in R.\]
Then
\begin{align*}
\operatorname{Ann}_R(F)&=\operatorname{Ann}_R(c_1^{-1}F)\\
&=(x_1^{a_1+b_1+1}, \dots, x_m^{a_m+b_m+1}, p',x_{m+2}^{a_{m+2}+1},\cdots, x_{m+n}^{a_{m+n}+1})\\
&=(x_1^{a_1+b_1+1}, \dots, x_m^{a_m+b_m+1}, p,x_{m+2}^{a_{m+2}+1},\cdots, x_{m+n}^{a_{m+n}+1})
\end{align*}
by Proposition \ref{prop c.i n=1}(2) and Lemma \ref{lem c.i add variable}.
\end{proof}

\section{the strong Lefschetz property}
In this section, we prove that $R/\operatorname{Ann}_R(F)$ has the strong Lefschetz property for a homogeneous binomial $F$ if $\operatorname{char}k=0$ and $R/\operatorname{Ann}_R(F)$ is a complete intersection.

\begin{defn}
Let $A$ be a graded Artinian algebra over $k$.
$A$ has the strong Lefschetz property if there exists $z\in A_1$ such that  the multiplication map $\times z^d: A_i \to A_{i+d}$ has maximal rank for any $i,d\in \mathbb Z_{\ge 0}$.
\end{defn}

\begin{prop}
Let $R=k[x_1,\dots,x_N]$ and $S=k[X_1,\dots,X_N]$ be  polynomial rings.
Let $F\in S$ be a nonzero homogeneous binomial.
Suppose that $\operatorname{char} k=0$ and $R/\operatorname{Ann}_R(F)$ is a complete intersection.
Then $R/\operatorname{Ann}_R(F)$ has the strong Lefschetz property.
\end{prop}

\begin{proof}
By Theorem \ref{main thm c.i}(1), it is enough to prove that $R/\operatorname{Ann}_R(F)$ has the strong Lefschetz property for $\operatorname{Ann}_R(F)$ in Theorem \ref{main thm c.i}(2)(3).
Hence we may assume that 
\[
R=k[x_1,\dots,x_{m+n}],\quad  S=k[X_1,\dots,X_{m+n}]
\]
\[F=X_1^{a_1}\cdots X_{m+n}^{a_{m+n}}(c_1X_1^{b_1}\cdots X_m^{b_m}-c_2X_{m+1}^{b_{m+1}})\in S,\] 
where $a_1,\dots,a_{m+n},b_1\dots,b_{m+1}\in \mathbb Z_{\ge 0}$ with $b_{m+1}\ge 1$ and $c_1,c_2\in k\setminus\{0\}$.
Since $F$ is homogeneous, we have $b_{m+1}=\sum_{i=1}^mb_i$.
Therefore $\operatorname{Ann}_R(F)$ is generated by  at most two homogeneous elements and several elements of the form $x_i^j$ by Theorem \ref{main thm c.i}(2)(3).
Hence $R/\operatorname{Ann}_R(F)$ has the strong Lefschetz property by \cite[Propsition 4.25.3]{HMMNWW}.
\end{proof}


\end{document}